\documentclass[11pt]{amsart}
\usepackage{fancybox}
\usepackage{graphicx}
\usepackage{ascmac}
\usepackage{tikz-qtree}
\usepackage{forest}
\usepackage{mathrsfs}
\usepackage{amscd,verbatim}
\usepackage{floatflt}
\usepackage{amssymb}
\usepackage{tikz-cd}
\usepackage[all]{xy}

\usepackage[hdivide={20mm,,20mm}, vdivide={20mm,,20mm}]{geometry}
\usepackage{stmaryrd} 
\usepackage[colorlinks,linkcolor=blue,citecolor=blue,urlcolor=red]{hyperref}
\usepackage[normalem]{ulem}
\usepackage{mathtools}

\newcommand{\bC}{\mathbb{C}}

\newcommand{\Q}{\mathbb{Q}}
\newcommand{\Z}{\mathbb{Z}}

\newcommand{\sC}{\mathcal{C}}
\newcommand{\sD}{\mathcal{D}}

\newcommand{\sO}{\mathcal{O}}

\newcommand{\sT}{\mathcal{T}}

\newcommand{\bS}{\mathbb{S}}

\newcommand{\fM}{\mathfrak{M}}

\newcommand{\Mod}{\operatorname{Mod}}
\newcommand{\free}{\operatorname{{free}}}
\newcommand{\perf}{\operatorname{perf}}
\newcommand{\cry}{\operatorname{cry}}

\newcommand{\LK}{{L_{K(1)}}}

\newcommand{\THH}{\operatorname{THH}}
\newcommand{\TP}{\operatorname{TP}}

\newcommand{\TCn}{\operatorname{TC}^{-}}
\newcommand{\HP}{\operatorname{HP}}

\newcommand{\Hom}{\operatorname{Hom}}

\newcommand{\Spec}{\operatorname{Spec}}

\newcommand{\tr}{{\operatorname{tr}}}

\newcommand{\et}{{\operatorname{\acute{e}t}}}

\renewcommand{\lim}{\operatornamewithlimits{\varprojlim}}

\newcommand{\ol}{\overline}

\renewcommand{\epsilon}{\varepsilon}

\newcommand{\Bcry}{B_{\operatorname{cry}}}

\newcommand{\Acry}{A_{\operatorname{cry}}}

\newcommand{\Ainf}{A_{\operatorname{inf}}}

\newcommand{\Ainfmu}{A_{\operatorname{inf}}[\frac{1}{ \mu }]}

\newcommand{\can}{\operatorname{can}}

\newcommand{\Cu}{\mathcal{C}}

\newcommand{\mS}{\mathfrak{S}}

\newcommand{\Cat}{\operatorname{Cat_{\infty}^{perf}}}
\newcommand{\Catsat}{\operatorname{Cat_{\infty,sat}^{perf}}}
\newcommand{\Catcsat}{\operatorname{Cat_{\infty,sat}^{c,perf}}}

\newcounter{spec}
{\end{list}}%

\newtheorem{lemma}{Lemma}[section]

\newcommand{\Sp}{\operatorname{Sp}}

\newtheorem{thm}[lemma]{Theorem}

\newtheorem{prop}[lemma]{Proposition}

\newtheorem{cor}[lemma]{Corollary}

\newtheorem{conj}[lemma]{Conjecture}

\theoremstyle{definition}

\newtheorem{defn}[lemma]{Definition}

\theoremstyle{remark}

\newtheorem{remark}[lemma]{Remark}

\numberwithin{equation}{section}

\newtheorem{notation}[lemma]{Notation}

\newcommand{\mf}{\mathfrak}

\title[$p$-adic Hodge theory]{Towards the $p$-adic Hodge theory for non-commutative algebraic varieties}\author{Keiho Matsumoto}

\begin{document}
\maketitle

\begin{abstract}
We construct a K-theory version of Bhatt-Morrow-Scholze's Breuil-Kisin cohomology theory for $\sO_K$-linear idempotent-complete, small smooth proper stable $\infty$-categories, where $K$ is a discretely valued extension of $\Q_p$ with perfect residue field. As a corollary, under the assumption that $K(1)$-local K theory satisfies the K\"unneth formula for $\sO_K$-linear idempotent-complete, small smooth proper stable $\infty$-categories, we prove a comparison theorem between $K(1)$-local K theory of the generic fiber and topological cyclic periodic homology theory of the special fiber with $\Bcry$-coefficients, and $p$-adic Galois representations of $K(1)$-local K theory for $\sO_K$-linear idempotent-complete, small smooth proper stable $\infty$-categories are semi-stable. We also provide an alternative K-theoretical proof of the semi-stability of p-adic Galois representations of the p-adic \'etale cohomology group of smooth proper varieties over $K$ with good reduction. This is a short, preliminary version of the work that was later expanded in \cite{matsumoto2023crystalline}.
\end{abstract}
\section{Introduction}
\begin{notation}
Fix a prime $p$. Let $K$ be a discretely valued field whose residue field $k=\sO_K/\pi$ is perfect of characteristic $p$. Here, $\sO_K$ is the ring of integers of $K$ and $\pi\in \sO_K$ is a uniformizer. We write $\Cu$ to denote the completion $\hat{\overline{K}}$ of $\overline{K}$ endowed with its unique absolute value extending the given absolute value on $K$, let $W$ be the Witt ring of $k$, and let $K_0$ be the fraction field of $W$. Let $\mf{m}$ be the maximal ideal of $\sO_K$. For a spectra $S$, let $L_{K(1)}S$ be the Bousfield localization of complex $K$ theory at prime $p$.
\end{notation}

\subsection{Background} For a $\Q_p$-vector representation $V$ of $G_K$, for $\bullet\in \{\text{crys, st}\}$, denote $D_\bullet(V)$ the $K_0$-vector space $(B_\bullet\otimes_{\Q_p}V)^{G_K}$. There exists an inequation (see \cite{Fontainecrys} and \cite{Fontainesemi})
\[
\dim_{K_0}D_{\text{crys}}(V) \leq \dim_{K_0}D_{\text{st}}(V) \leq \dim_{\Q_p} V.
\]
We say that $V$ is crystalline (resp. semi-stable) if $\dim_{K_0}D_{\text{crys}}(V)=\dim_{\Q_p} V$ (rsep. $\dim_{K_0}D_{\text{st}}(V) = \dim_{\Q_p} V$). For a variety $X$ over $\sO_K$, the absolute Galois group $G_K:=\text{Gal}(\ol{K}/K)$ acts on $p$-adic \'etale cohomology theory $H^i_{\et}(X_\sC,\Q_p)$ and $K(1)$-local $K$ theory $\pi_*\LK K(X_\sC)$. We then get the following result.
\begin{lemma}
 We assume $X_\sC$ is smooth over $\sC$. $H^i_{\et}(X_\sC,\Q_p)$ is crystalline (resp. semi-stable) for any $i$ if and only if $\pi_j\LK K(X_\sC)$ is crystalline (resp. semi-stable) for $j=0,1$.
\end{lemma}
\begin{proof}
Due to Thomason \cite{Thomason}, there exists a spectral sequence (see \cite[(0.1)]{Thomason})
\begin{equation}\label{bakido}
E_2^{i,j}=\left\{
\begin{array}{ll}
H^{i}_\et(X_\sC,\Z_p(l)) & j = 2l \\
0 & j \text{ odd} 
\end{array}
\right. \Longrightarrow \pi_{j-i} L_{K(1)}K(X_\sC).
\end{equation} 
This spectral sequence degenerates after tensoring $\mathbb{Q}_p$ and we have a $G_K$-equivariant isomorphism
\begin{eqnarray}
  \bigoplus_{i\in\Z} H_\et^{2i}(X_\sC,\Q_p(i)) &\simeq & \pi_0L_{K(1)}K(X_\sC)\label{even02}\\
   \bigoplus_{i\in\Z} H_\et^{2i-1}(X_\sC,\Q_p(i)) &\simeq &\pi_1L_{K(1)}K(X_\sC)\label{odd03}.
\end{eqnarray}
For a $\Q_p$-vector representation $V$ of $G_K$ and a natural number $m$, $V$ is crystalline (resp. semi-stable) if and only if $V\otimes_{\Q_p} \Q_p(m)$ is crystalline (resp. semi-stable) (see \cite[5.2 Theorem i)]{Fontainecrys} and \cite[Proposition 1.5.2, Proposition 5.1.2]{Fontainesemi}). A direct sum of crystalline (resp. semi-stable) representations is also crystalline (resp. semi-stable), and a direct summand of crystalline (resp. semi-stable) representations is also crystalline (resp. semi-stable), therefor we obtain the claim.
\end{proof}

This lemma suggests that $K(1)$-local K-theory can be used as an alternative to the $p$-adic \'etale cohomology theory for investigating the properties of p-adic Galois representations of algebraic varieties. K(1)-local K theory can also be defined for a stable $\infty$-category. On the other hand, it has been shown that the integral $p$-adic \'etale cohomology theory cannot be defined for a stable $\infty$-category (see \cite{fundamentalisnotderived}). Let us apply this perspective to other cohomology theories and variants of K-theory.

 In the study of complex manifolds, de Rham theory states that for a complex manifold $M$, the de Rham cohomology $H^i_{\text{dR}}(M)$ of $M$ is naturally isomorphic to its singular cohomology $H^i_{\text{Sing}}(M,\bC)$ with coefficients in $\mathbb{C}$. Blanc predicts the non-commutative version of de Rham theory \cite[Conejcture 1.7]{Blanc}.
\begin{conj}[The lattice conjecture, \cite{Blanc}]\label{lattice}
Let $\sT$ be a smooth proper dg-category over C. Then the map 
\[
\operatorname{Ch}^{\operatorname{top}}\wedge_{\bS}H\bC:K_{\operatorname{top}}(\sT)\wedge_{\bS}H\bC \to \HP(\sT/\bC)
\]
is an equivalence. In particular, there is an isomorphism 
\[
\pi_iK_{\operatorname{top}}(\sT) \otimes_\Z \bC \simeq \HP_i(\sT/\bC).
\]
for any $i$.
\end{conj}
In the study of $p$-adic cohomology theories, crystalline comparison theory \cite{Falcrys} states that for a smooth proper variety $X$ over $\sO_K$, the $p$-adic \'etale cohomology $H^i_\et(X_\sC,\Q_p)\otimes_{\Q_p}\Bcry$ is isomorphic to $H^i_{cry}(X_\kappa/W)\otimes_{W}\Bcry$, moreover the isomorphism is compatible with $G_K$-action, Frobenius endomorphism and filtration. We study non-commutative version of the crystalline comparison theorem. Inspired by Petrov and Vologodsky's work on non-commutative crystalline cohomology theory \cite{petrov}, and the study of motivic filtration of the $K(1)$-K theory \cite{Thomason} and the topological periodic homology theory \cite{BMS2}, we predict the following conjecture.
\begin{conj}[Non-commutative version of crystalline comparison theorem~\cite{Falcrys}] \label{NCcryconj}
Let $\sT$ be an $\sO_K$-linear idempotent-complete, small smooth proper stable infinity category. Then there is an isomorphism of $\Bcry$-module:
\[
\pi_i\TP(\sT_k;\Z_p)\otimes_W \Bcry \simeq \pi_i \LK K(\sT_\sC)\otimes_{\Z_p}\Bcry
\]
which is compatible with $G_K$-action, Frobenius endomorphism and filtration. In particular, the $p$-adic representation $\pi_i\LK K(\sT_\sC)\otimes_{\Z_p}\Q_p$ is crystalline.
\end{conj}
This paper approaches this conjecture via $K$-theoritical version of Bhatt-Morrow-Scholze's comparison theorem \cite{BMS2}. We will use the language of stable $\infty$-categories, following Lurie \cite{Lurie17}.
\begin{defn}
 For an $\mathbb{E}_\infty$-ring $R$, we let $\Cat (R)$ denote the $\infty$-category of $R$-linear idempotent-complete, small stable $\infty$-categories, where the morphisms are exact functors. 
\end{defn}

\begin{defn}
For an $\mathbb{E}_\infty$-ring $R$, we let $\Catsat (R)$ denote the $\infty$-category of $R$-linear idempotent-complete, small smooth proper stable $\infty$-categories, where the morphisms are exact functors. 
\end{defn}
\begin{defn}
  For a commutative ring $R$, we let $\Catcsat (R)$ denote the $\infty$-category of $R$-linear idempotent-complete, small smooth proper stable $\infty$-categories of perfect complex on smooth proper varieties over $R$, where the morphisms are exact functors.
\end{defn}
\begin{remark}
For an $\mathbb{E}_\infty$-ring $R$, $\Cat (R)$ acquires the structure of a symmetric monoidal $\infty$-category (see \cite[Appendix D.2.2]{Lurie17}), and $\Catsat (R)$ also acquires the structure of a symmetric monoidal $\infty$-category. For a smooth proper variety $X$ over a 
 commutative ring $R$, the smooth proper stable $\infty$-category $\perf(X)$ is self-dual in $\Catcsat (R)$ (see \cite[Corollary 9.4.3.4.]{Lurie17}), and $\Catcsat (R)$ acquires the structure of a symmetric monoidal $\infty$-category (see \cite[Corollary 9.4.3.8]{Lurie17}).
\end{remark}
To approach the non-commutative version of the crystalline conjecture, We will be interested in the following conjecture, which closely relates to conjecture~\ref{lattice}. 
\begin{conj}\label{kunnethK1}
A lax symmetric monoidal $\infty$-functor
\[
L_{K(1)}K(-):\Catsat(\sC) \to {\operatorname{Mod}}_{L_{K(1)}K(\sC)}(\Sp)
\]
is actually symmetric monoidal, i.e. for any $\sC$-linear idempotent-complete, small smooth proper stable $\infty$-categories $\sT_1,\sT_2$, the natural map 
\[
L_{K(1)}K(\sT_1) \otimes_{L_{K(1)}K(\sC)} L_{K(1)}K(\sT_2) \to L_{K(1)}K(\sT_1\otimes_\sC \sT_2)
\]
is an equivalence.
\end{conj}
 The conjecture holds for commutative algebraic varieties when $K/\Q_p$ is a finite extension.
\begin{prop}\label{commkunnethK1}
  We assume $K/\Q_p$ is a finite extension. For smooth proper varieties $X_1,X_2$ over $\sC$, the natural map 
\[
L_{K(1)}K(X_1) \otimes_{L_{K(1)}K(\sC)} L_{K(1)}K(X_2) \to L_{K(1)}K(X_1\times_{\Spec \sC} X_2)
\]
is an equivalence.
\end{prop}
\begin{proof}
We fix an isomorphism of fields $\sigma:\sC \simeq \bC$. We denote $X_{i}\times_{\Spec \sC} \Spec \bC$ by $X_{i,\bC}$. Due to Thomason \cite{Thomason}, we have an equivalence $L_{K(1)}K(X_i)\simeq K_{\text{top}}(X_{i,\bC}^{an})^{\wedge}_p$. Due to Atiyah \cite{Atiyah}, since a CW complex which comes from a compact complex manifold is finite, we know a natural map
\[
K_{\text{top}}(X_{1,,\bC}^{an}) \otimes_{K_{\text{top}}(\text{pt})} K_{\text{top}}(X_{2,\bC}^{an}) \to K_{\text{top}}(X_{1,\bC}^{an}\times X_{2,\bC}^{an})
\]
is an equivalence. The isomorphism of fields $\sigma:\sC \simeq \bC$ induces an equivalence of ring spectra $K_{\text{top}}(\text{pt})^{\wedge}_p \simeq \LK K(\bC)$. We have a diagram
\[
\xymatrix{
L_{K(1)}K(X_1) \otimes_{L_{K(1)}K(\sC)} L_{K(1)}K(X_2) \ar[d]_-{\simeq} \ar[r]&L_{K(1)}K(X_1\times_{\Spec \sC} X_2)\ar[d]^{\simeq}\\
K_{\text{top}}(X_{1,,\bC}^{an}) \otimes_{K_{\text{top}}(\text{pt})} K_{\text{top}}(X_{2,\bC}^{an})\ar[r]_-{\simeq}&K_{\text{top}}(X_{1,\bC}^{an}\times X_{2,\bC}^{an})
}
\]
and we obtain the claim.
\end{proof}

 Fix a sequence of elements $\zeta_n\in K$ inductively such that $\zeta_0=1$ and $(\zeta_{n+1})^p=\zeta_n$, let $\epsilon=(\zeta_0,\zeta_1,\zeta_2,...)\in \lim_{\text{Frob}}\sO_\Cu/p$, and let $\epsilon=(\zeta_0,\zeta_1,\zeta_2,...)\in \lim_{\text{Frob}}\sO_\Cu/p$, and let $[\epsilon]\in \Ainf$ be the Teichm\"uller lift. Let $\mS=W[[z]]$, and let $\tilde{\theta}:\mS \to \sO_K$ be the usual map whose kernel is generated by Eisenstein polynomial $E$ of $\pi$. We choose $\xi \in \Ainf$ of the kernel of the canonical map $\Ainf\to \sO_\sC$. We write $\mu=[\epsilon]-1$. Let $\varphi:\mS \to \mS$ be a Frobenius endomorphism which is Frobenius on $W$ and sends $z$ to $z^p$. In this paper, we prove the following.

\begin{thm}[{Theorem~\ref{NCBMS}}, {Non-commutative version of \cite{BMS2}}]\label{NCBMSintro}
  Let $\sT$ be an $\sO_K$-linear idempotent-complete, small smooth proper stable $\infty$-category. Then there is a natural number $n$ satisfying the following holds:
 \begin{itemize}
   \item[(1)] For any $i\geq n$, $\pi_i \TCn(\sT/\bS[z];\Z_p)$ is a Breuil-Kisin module.
   \item[(2)] (K(1)-local K theory comparison) Assume Conjecture \ref{kunnethK1}. For any $i\geq n$, after scalar extension along $\ol{\phi}:\mS\to \Ainf$ which sends $z$ to $[\pi^{\flat}]^p$ and is the Frobenius on $W$, one recovers K(1)-local K theory of the generic fiber
   \[
   \pi_i \TCn(\sT/\bS[z];\Z_p) \otimes_{\mS,\ol{\phi}} \Ainf[\frac{1}{\mu}]^{\wedge}_p \simeq \pi_i\LK K(\sT_{\Cu})\otimes_{\Z_p} \Ainf[\frac{1}{\mu}]^{\wedge}_p
   \]
   \item[(3)] (topological periodic homology theory comparison) For any $i\geq n$, after scalar extension along the map $\tilde{\phi}:\mS\to W$ which is the Frobenius on $W$ and sends $z$ to $0$, one recovers topological periodic homology theory of the special fiber:
  \[
   \pi_i \TCn(\sT/\bS[z];\Z_p)[\frac{1}{p}] \otimes_{\mS[\frac{1}{p}],\tilde{\phi}} K_0 \simeq \pi_i \TP(\sT_k;\Z_p)[\frac{1}{p}].
   \]
   \item[(4)] If there is a smooth proper variety $X$ such that $\sT=\perf(X)$, then for amy $i\geq n$, there is an isomorphism
   \[
    \pi_i \TCn(X/\bS[z];\Z_p) \otimes_{\mS,\ol{\phi}} \Ainf[\frac{1}{\mu}]^{\wedge}_p \simeq \pi_i\LK K(X_{\Cu})\otimes_{\Z_p} \Ainf[\frac{1}{\mu}]^{\wedge}_p
  \]
  \end{itemize}
\end{thm}
\begin{cor}[{Corollary~\ref{Bcrycor}}]\label{maincor}
We assume the conjecture~\ref{kunnethK1} holds. For an $\sO_K$-linear idempotent-complete, small smooth proper stable $\infty$-category $\sT$ and an integer $i$, there is an isomorphism of $\Bcry$-modules:
\[
\pi_i\TP(\sT_k;\Z_p)\otimes_W \Bcry \simeq \pi_i L_{K(1)}K(\sT_\Cu)\otimes_{\Z_p} \Bcry.
\]
\end{cor}
\begin{thm}[{Theorem~\ref{mainwithproof}}, Main theorem]\label{main}
\begin{itemize}
 \item[(1)] For an $\sO_K$-linear idempotent-complete, small smooth proper stable $\infty$-category $\sT$, there is a natural number $n$ such that for any $i\geq n$, $\pi_i \TCn(\sT/\bS[z];\Z_p)_{\free}$ is a Breuil-Kisin $G_K$-module in the sense of \cite{Gao}. In particular, $\bigl( \widehat{\pi_i\TCn(\sT/\bS[z];\Z_p)_{\free}}\otimes_{\Ainf} W(\Cu) \bigr)^{\widehat{\varphi}\otimes\varphi_{\Cu}=1}$ is a $\Z_p$-lattice of a semi-stable representation, where $\varphi$ is the Frobenius map of $\TCn(\sT/\bS[z];\Z_p)_{\free}$ and $\widehat{\varphi}$ is given in \eqref{power} and $\varphi_{\Cu}$ is the Frobenius of $W(\Cu)$.

  \item[(2)] We assume the conjecture~\ref{kunnethK1} holds. Then for an $\sO_K$-linear idempotent-complete, small smooth proper stable $\infty$-category $\sT$ and an integer $i$, there is a $G_K$-equivariant isomorphism
\[
\pi_i \LK K(\sT_{\Cu})\otimes_{\Z_p}\Q_p \simeq \bigl(\bigl( \widehat{\pi_i\TCn(\sT/\bS[z];\Z_p)_{\free}}\otimes_{\Ainf} W(\Cu) \bigr)^{\widehat{\varphi}\otimes \varphi_{\Cu}=1}\otimes_{\Z_p}\Q_p.
\]
In particular, $\pi_i \LK K(\sT_{\Cu})\otimes_{\Z_p}\Q_p$ is a semi-stable representation.

\item[(3)] For a smooth proper variety $X$ over $\sO_K$, $H^*_\et(X_{\Cu},\Q_p)$ is a semi-stable representation. 
\end{itemize}
\end{thm}
\begin{remark}
Statement (3) is directly deduced from the crystalline comparison theorem. In this paper, we will give a $K$-theoretical proof.
\end{remark}
\begin{remark}
  Bhatt-Morrow-Scholze's Breuil-Kisin cohomology theory \cite[Theorem~1.2]{BMS2} implies crystalline comparison theorem \cite{Falcrys}. On the other hand, even assuming Conjecture~\ref{kunnethK1}, Theorem~\ref{NCBMSintro} does not imply Conjecture~\ref{NCcryconj}. This difference arises as follows. On Breuil-Kisin cohomology theory, there is a $G_K$-equivariant isomorphism (see \cite[Theorem 1.8 (iii)]{BMS1})
  \begin{equation}\label{diffCNC}
  R\Gamma_{\Ainf}(X_{\sO_\Cu})\otimes_{\Ainf} \Acry \simeq R\Gamma_{\cry}(X_{\sO_\Cu/p}/\Acry).
  \end{equation}
Combine \eqref{diffCNC} with \cite[Theorem~1.8 (iv)]{BMS1}, there is canonical $(G_K,\varphi)$-equivariant isomorphism 
\begin{equation}\label{BMSmulocal}
R\Gamma_{\et}(X_{\Cu},\Z_p)\otimes_{\Z_p} \Acry[\frac{1}{p\mu}] \simeq R\Gamma_{\cry}(X_{\sO_\Cu/p}/\Acry)[\frac{1}{p\mu}].
\end{equation}
This induces the crystalline comparison theorem (see \cite[Theorem 14.4]{BMS1}). On the other hand, the non-commutative analogue of \eqref{diffCNC} becomes the following $G_K$-equivariant isomorphism
 \begin{equation}\label{diffCNC2}
 \pi_i \TP(\sT_{\sO_\Cu};\Z_p)\otimes_{\Ainf} \widehat{\Acry} \simeq \pi_i\TP(\sT_{\sO_\Cu/p};\Z_p),
 \end{equation}
 where $\widehat{\Acry}$ is the completion of $\Acry$ with respect to the Nygaard filtration (see \cite[Definition 8.9]{BMS2}). The problem is that $\mu$ is $0$ in $\widehat{\Acry}$ (see \cite[Corollary 2.11 and Corollary 2.12]{remarkonK1}). Thus the non-commutative analogue of \eqref{BMSmulocal} becomes the trivial equitation
 \[
 0=\pi_i \LK K(\sT_\Cu)\otimes_{\Z_p} \widehat{\Acry}[\frac{1}{p\mu}] \simeq \pi_i\TP(\sT_{\sO_\Cu/p};\Z_p)[\frac{1}{p\mu}]=0.
 \]
\end{remark}

\section*{Acknowledgement}
I would like to thank Federico Binda, Lars Hesselholt, Shane Kelly, Hyungseop Kim, Hiroyasu Miyazaki for helpful discussions related to this subject. I would also like to thank Ryomei Iwasa for comments on an earlier draft and Alexender Petrov for useful comments on a draft.

\section{Non-commutative version of Breuil-Kisin cohomology}


\subsection{Breuil-Kisin modules and Breuil-Kisin cohomology theory $R\Gamma_{\mS}$} Let us start by recalling the theory of Breuil–Kisin modules.
\begin{defn}
  A Breuil–Kisin module is a finitely generated $\mS$-module $M$ equipped with an isomorphism
  \[
  \varphi_M:M\otimes_{\mS,\varphi} \mS[\frac{1}{E}] \simeq M[\frac{1}{E}].
  \]
\end{defn}
\begin{lemma}[{\cite[Corollaire 11.1.14]{FarguesFontaine}} and {\cite[Lemma 4.27]{BMS1}}]\label{BKlemma} Let $\ol{\phi}:\mS \to \Ainf$ be the map that sends $z$ to $[\pi^\flat]^p$ and is Frobenius on $W$. Let $M$ be a Breuil–Kisin module, and let $M_{\Ainf}=M\otimes_{\mS,\ol{\phi}}\Ainf$. Then $M'[\frac{1}{p}]=M_{\Ainf}[\frac{1}{p}]\otimes_{\Ainf[\frac{1}{p}]}K_0$ is a finite free $K_0$-module equipped with a Frobenius automorphism. Fix a section $k \to \sO_{\Cu}/p$, then here is a (noncanonical) $\varphi$-equivariant isomorphism
\[
M_{\Ainf}\otimes_{\Ainf}\Bcry \simeq M'[\frac{1}{p}]\otimes_{K_0}\Bcry.
\]
\end{lemma}
In \cite{BMS2}, Bhatt-Morrow-Scholze constructed a cohomology theory valued in Breuil–Kisin modules for smooth proper varieties over $\sO_K$. Let $\varphi_{\Ainf}:\Ainf \to \Ainf$ be the Frobenius endomorphism of $\Ainf$. Let $\phi:\mS \to \Ainf$ be the $W$-linear map that sends $z$ to $[\pi^\flat]$. Note that the diagram 
\begin{equation}\label{SAinfdiag}
\xymatrix
{\mS \ar[r]^{\varphi} \ar[d]_{\phi} \ar[rd]^{\ol{\phi}}& \mS \ar[d]^{\phi} \\
\Ainf \ar[r]_{\varphi_{\Ainf}}^{\simeq } & \Ainf
}
\end{equation}
commutes.
\begin{thm}[{\cite{BMS2}}]\label{BKcoh}
 Let $X/\sO_K$ be a proper smooth formal scheme. Then there exists a $\mS$-linear cohomology theory $R\Gamma_{\mS}(X)$ equipped with a $\varphi$-semilinear map, with the following properties:
 \begin{itemize}
   \item[(1)] All $H^i_\mS(X) :=H^i(R\Gamma_\mS(X))$ are Breuil-Kisin modules.
   \item[(2)] (\'etale comparison) After scalar extension along $\ol{\phi}:\mS\to \Ainf$, one recovers \'etale cohomology of the generic fiber
   \[
   R\Gamma_{\mS}(X) \otimes_{\mS} \Ainfmu \simeq R\Gamma_{\et}(X_\Cu,\Z_p)\otimes_{\Z_p} \Ainfmu 
   \]
   \item[(3)] (crystalline comparison) After scalar extension along the map $\mS\to W$ which is the Frobenius on $W$ and sends $z$ to $0$, one recovers crystalline cohomology of the special fiber:
   \[
   R\Gamma_\mS(X)\otimes_{\mS}^{\mathbb{L}}W \simeq R\Gamma_{\cry}(X_k/W).
   \]
   \item[(4)] (de Rham comparison) After scalar extension along the map $\tilde{\theta}\circ \varphi:\mS \to \sO_K$, one recovers de Rham cohomology:
   \[
   R\Gamma_{\mS}(X)\otimes_{\mS}^{\mathbb{L}}\sO_K \simeq R\Gamma_{\text{dR}}(X/\sO_K).
   \]
   \end{itemize}
\end{thm}

\subsection{Perfect modules and K\"unneth formula} Let $(\mathcal{A},\otimes,1_{\mathcal{A}})$ be a symmetric monoidal, stable $\infty$-category with biexact tensor product. Firstly we recall from \cite[section 1]{KunnethTP}.
\begin{defn}[{\cite[Definition 1.2]{KunnethTP}}]
  An object $X\in A$ is perfect if it belongs to the thick subcategory generated by the unit.
\end{defn}
For a lax symmetric monoidal, exact $\infty$-functor $F:\mathcal{A} \to \Sp$, $F(1_{\mathcal{A}})$ is naturally an $\mathbb{E}_\infty$-ring. For any $X,Y\in \mathcal{A}$, we have a natural map
\begin{equation}\label{laxmon}
F(X)\otimes_{F(1_{\mathcal{A}})} F(Y) \to F(X\otimes Y).
\end{equation}
Since $F$ is exact, if $X$ is perfect, then the map \eqref{laxmon} is an equivalence, and $F(X)$ is perfect as an $F(1_{\mathcal{A}})$-module. 

We regard $\sO_K$ as an $\bS [z]$-algebra via $z\mapsto \pi$. There is a symmetric monooidal $\infty$-functor
\[
\THH(-/\bS[z];\Z_p):\Cat (\sO_K) \to \Mod_{\THH(\sO_K/\bS[z];\Z_p)}(\Sp^{BS^1}).
\]
Let us study this functor from \cite[section 11]{BMS2}. By the base change along $\bS[z] \to \bS[z^{1/{p^\infty}}]~|~z\mapsto z$, there is a natural equivalence
\begin{equation}\label{basechange}
\THH(\sO_K/\bS[z]) \otimes_{\bS[z]} \bS[z^{1/{p^\infty}}] \simeq \THH(\sO_K[\pi^{1/p^{\infty}}]/\bS[z^{1/{p^\infty}}]).
\end{equation}
The natural map $\THH(\bS[z^{1/{p^{\infty}}}];\Z_p) \to \bS[z^{1/{p^{\infty}}}]^{\wedge}_p$ is an equivalence (see \cite[Proposition. 11.7]{BMS2}), we obtain equivalences
\begin{equation}
\THH(\sO_K[\pi^{1/p^{\infty}}];\Z_p) \simeq \THH(\sO_K[\pi^{1/p^{\infty}}]/\bS[z^{1/{p^\infty}}];\Z_p) \overset{\eqref{basechange}}{\simeq}\THH(\sO_K/\bS[z];\Z_p) \otimes_{\bS[z]^{\wedge}_p} \bS[z^{1/{p^\infty}}]^{\wedge}_p.
\end{equation}
Since a morphism of homotopy groups $\pi_*(\bS[z]^{\wedge}_p) \to \pi_*(\bS[z^{1/{p^\infty}}]^{\wedge}_p)$ is faithfully flat and there is an isomorphism
$\pi_*\THH(\sO_K[\pi^{1/p^{\infty}}];\Z_p) \simeq \sO_K[\pi^{1/p^{\infty}}][u]$, where $\deg u=2$ (see \cite[section. 6]{BMS2}), we obtain an isomorphism
\begin{equation}\label{BMScalcu}
\pi_*\THH(\sO_K/\bS[z];\Z_p) \simeq \sO_K[u],
\end{equation}
where $u$ has degree $2$ (see \cite[Proposition. 11.10]{BMS2}).
\begin{prop}\label{propdualisperf}
Any dualizable object in $\Mod_{\THH(\sO_K/\bS[z];\Z_p)}(\Sp^{BS^1})$ is perfect.
\end{prop}
\begin{proof}
By the isomorphism \eqref{BMScalcu}, $\pi_*\THH(\sO_K/\bS[z];\Z_p)$ is a regular noetherian ring of finite Krull dimension concentrated in even degrees. We obtain the claim by \cite[Theorem 2.15]{KunnethTP}.
\end{proof}
\begin{prop}\label{TPTCkunneth}
Let $\sT_1,\sT_2$ be $\sO_K$-linear idempotent-complete, small stable $\infty$-categories and suppose $\sT_1$ is smooth and proper. Then $\THH(\sT_1/\bS[z];\Z_p)$ is perfect module $\Mod_{\THH(\sO_K/\bS[z];\Z_p)}(\Sp^{BS^1})$, and $\TCn(\sT_1/\bS[z];\Z_p)$ (resp. $\TP(\sO_K/\bS[z];\Z_p)$) is a perfect $\TCn(\sT_1/\bS[z];\Z_p)$-module (resp. $\TP(\sO_K/\bS[z];\Z_p)$-module) and the natural map
\[
\TCn(\sT_1/\bS[z];\Z_p) \otimes_{\TCn(\sO_K/\bS[z];\Z_p)}\TCn(\sT_2/\bS[z];\Z_p) \to \TCn(\sT_1\otimes_{\sO_K}\sT_2/\bS[z];\Z_p)
\]
\[
(\text{resp. }\TP(\sT_1/\bS[z];\Z_p) \otimes_{\TP(\sO_K/\bS[z];\Z_p)}\TP(\sT_2/\bS[z];\Z_p) \to \TP(\sT_1\otimes_{\sO_K}\sT_2/\bS[z];\Z_p) \quad )
\]
is an equivalence.
\end{prop}
\begin{proof}
Note that $\infty$-functors $(-)^{hS^1}:\Mod_{\THH(\sO_K/\bS[z];\Z_p)}(\Sp^{BS^1}) \to \Mod_{\TCn(\sO_K/\bS[z];\Z_p)}(\Sp)$ and $(-)^{tS^1}:\Mod_{\THH(\sO_K/\bS[z];\Z_p)}(\Sp^{BS^1}) \to \Mod_{\TCn(\sO_K/\bS[z];\Z_p)}(\Sp)$ are lax symmetric monoidal, exact (see \cite[Corollary I.4.3]{NS18}). It is enough to show $\THH(\sT_1/\bS[z];\Z_p)$ is perfect in $\Mod_{\THH(\sO_K/\bS[z];\Z_p)}(\Sp^{BS^1})$. Since a $\infty$-functor
\[
\THH(-/\bS[z];\Z_p): \Cat(\sO_K) \to\Mod_{\THH(\sO_K/\bS[z];\Z_p)}(\Sp^{BS^1})
\]
is symmetric monoidal, and $\sT_1$ is dualizable in $\Cat(\sO_K)$ (Cf. \cite[Ch. 11]{Lurie17}). Thus $\THH(\sT_1/\bS[z];\Z_p)$ is dualizable in $\Mod_{\THH(\sO_K/\bS[z];\Z_p)}(\Sp^{BS^1})$, by Proposition \ref{propdualisperf}, we obtain the claim. 
\end{proof}

\subsection{Breuil-Kisin module of stable $\infty$-categories} Let us recall Antieau-Mathew-Nikolaus's comparison theorem of symmetric monoidal $\infty$-functors from \cite{KunnethTP}.
\begin{prop}[{\cite[Proposition 4.6]{KunnethTP}}]\label{compmonoid}
Let $\sT,\hat{\sT}$ be symmetric monoidal $\infty$-categories. Let $F_1,F_2:\sT\to\hat{\sT}$ be symmetric monoidal functors and let $t:F_1\Longrightarrow F_2$ be a symmetric monoidal natural transformation. Suppose every object of $\sT$ is dualizable. Then $t$ is an equivalence.
\end{prop}
In \cite[Proposition 11.10]{BMS2}, Bhatt-Morrow-Scholze showed the followings: on homotopy groups, 
\[
\pi_i \TCn(\sO_K/\bS[z];\Z_p) \simeq \mS[u,v]/(uv-E)
\]
where $u$ is of degree $2$ and $v$ is of degree $-2$, 
\[
\pi_i \TP(\sO_K/\bS[z];\Z_p) \simeq \mS[\sigma^{\pm}]
\]
where $\sigma$ is of degree $2$. Let $\varphi$ be the endomorphism of $\mS$ determined by the Frobenius on $W$ and $z\mapsto z^{p}$. Scholze-Nikolaus construct two maps
\[
\can,\varphi^{hS^1}_\sT:\TCn(\sT/\bS[z];\Z_p) \rightrightarrows \TP(\sT/\bS[z];\Z_p),
\]
where $\can_\sT$ and $\varphi^{hS^1}_\sT$ are constructed in \cite{NS18}. In \cite[Proposition 11.10]{BMS2}, Bhatt-Morrow-Scholze also showed the morphism 
\[
\can_{\sO_K}:\pi_*\TCn(\sO_K/\bS[z];\Z_p) \to \pi_*\TP(\sO_K/\bS[z];\Z_p)
\]
sends $u$ to $E\sigma$ and $v$ to $\sigma^{-1}$, and the morphism
\[
\varphi^{hS^1}_{\sO_K}:\pi_*\TCn(\sO_K/\bS[z];\Z_p) \to \pi_*\TP(\sO_K/\bS[z];\Z_p)
\]
is $\varphi$-linear map which send $u$ to $\sigma$ and $v$ to $\varphi(E) \sigma^{-1}$. Since $\sigma$ is an invertible element in $\TP(\sO_K/\bS[z];\Z_p)$, $\varphi^{hS^1}_{\sO_K}$ induces a map
\begin{equation}\label{localFrob}
\tilde{\varphi}^{hS^1}_{\sT}:\TCn(\sT/\bS[z];\Z_p)[\frac{1}{u}] \to \TP(\sT/\bS[z];\Z_p)
\end{equation}
for an $\sO_K$-linear stable $\infty$-category $\sT$. Note that on homotopy groups,
\[  \tilde{\varphi}^{hS^1}_{\sO_K}:\pi_*\TCn(\sO_K/\bS[z];\Z_p)[\frac{1}{u}] \to \pi_*\TP(\sO_K/\bS[z];\Z_p)
\]
is given by $\mS[u^{\pm}] \to \mS[\sigma^{\pm}]$ which is $\varphi$-semi-linear and sends $u$ to $\sigma$.
\begin{lemma}\label{periodlemma}
Let $\sT$ be an $\sO_K$-linear idempotent-complete, small smooth proper stable $\infty$-category. Then the followings hold:
\begin{itemize}
  \item[(1)] $\pi_*\TCn(\sT/\bS[z];\Z_p)$ is a finitely generated $\pi_*\TCn(\sO_K/\bS[z];\Z_p)$-module, thus all $\pi_i\TCn(\sT/\bS[z];\Z_p)$ is finitely generated $\mS$-module.
  \item[(2)] There is a natural number $n$ satisfying that for any $j \geq n$, $\pi_j\TCn(\sT/\bS[z];\Z_p)\overset{u\cdot }{\to} \pi_{j+2}\TCn(\sT/\bS[z];\Z_p)$ is an isomorphism. 
\end{itemize}
\end{lemma}
\begin{proof}
Claim (1) directly follows from Proposition ~\ref{TPTCkunneth}. For any $j \geq 0$, by the calculation of $\pi_* \TCn(\sO_K/\bS[z];\Z_p)$ (see \cite[Proposition 11.10]{BMS2}), we know $\pi_j\TCn(\sO_K/\bS[z];\Z_p)\overset{u\cdot}{\to} \pi_{j+2}\TCn(\sO_K/\bS[z];\Z_p)$ is an isomorphism. These yields claim (2) by claim (1).
\end{proof}
\begin{thm}\label{BKtheorem}
Let $\sT$ be an $\sO_K$-linear idempotent-complete, small smooth proper stable $\infty$-category. Then there is a natural number $n$ such that the homotopy groupy group $\pi_i \TCn(\sT/\bS[z];\Z_p)$ is a Breuil-Kisin module for any $i\geq n$, where Frobenius comes from cyclotomic Frobenius map.
\end{thm}
\begin{proof}
The morphism $\tilde{\varphi}^{hS^1}_\sT$ induces a morphism of $\TP(\sO_K/\bS[z];\Z_p)$-module:
\begin{equation}\label{TCTP}
\TCn(\sT/\bS[z];\Z_p) [\frac{1}{u}]\otimes_{\TCn(\sO_K/\bS[z];\Z_p)[\frac{1}{u}],\tilde{\varphi}^{hS^1}_{\sO_K}} \TP(\sO_K/\bS[z];\Z_p) \to \TP(\sT/\bS[z];\Z_p).
\end{equation}
By proposition~\ref{TPTCkunneth}, both sides of the map \eqref{TCTP} yields symmetric monoidal functors from $\Catsat(\sO_K)$ to $\Mod_{\TP(\sO_K/\bS[z];\Z_p)}(\Sp)$, and the map \eqref{TCTP} yields a symmetric monoidal natural transformation between them. By Proposition~\ref{compmonoid}, the morphism \eqref{TCTP} is an equivalence. On homotopy groups, $\tilde{\varphi}^{hS^1}_{\sO_K}$ is $\varphi:\mS\to \mS$ on each degree. Since $\varphi$ is flat, one has an isomorphism
\begin{equation}\label{preTCnKisin}
\pi_i\TCn(\sT/\bS[z];\Z_p)[\frac{1}{u}] \otimes_{\mS,\varphi} \mS \simeq \pi_i \TP(\sT/\bS[z];\Z_p)
\end{equation}
for any $i$. By Lemma~\ref{periodlemma} (2), one obtains an isomorphism 
\[
\pi_i\TCn(\sT/\bS[z];\Z_p) \simeq \pi_i\TCn(\sT/\bS[z];\Z_p)[\frac{1}{u}] 
\]
for any $i \geq n$, thus we have an isomorphism on homotopy groups:
\begin{equation}\label{TCnKisin}
\pi_i\TCn(\sT/\bS[z];\Z_p) \otimes_{\mS,\varphi} \mS \simeq \pi_i \TP(\sT/\bS[z];\Z_p)
\end{equation}
for any $i \geq n$.

After localization $E\in \mS \simeq \pi_0 \TCn(\sO_K/\bS[z];\Z_p)$, the morphism $\can_\sT$ induces a morphism of $\TP(\sO_K/\bS[z];\Z_p)[\frac{1}{E}]$-module:
\begin{equation}\label{TCTPcan}
\TCn(\sT/\bS[z];\Z_p)[\frac{1}{E}] \otimes_{\TCn(\sO_K/\bS[z];\Z_p)[\frac{1}{E}],\can_{\sO_K}} \TP(\sO_K/\bS[z];\Z_p)[\frac{1}{E}] \to \TP(\sT/\bS[z];\Z_p)[\frac{1}{E}].
\end{equation}
Both sides of the map \eqref{TCTPcan} yield symmetric monoidal functors from $\Catsat(\sO_K)$ to $\Mod_{\TP(\sO_K/\bS[z];\Z_p)[\frac{1}{E}]}(\Sp)$, and the map \eqref{TCTPcan} yield a symmetric monoidal natural transformation between them. Thus the morphism \eqref{TCTPcan} is an equivalence. Note that the morphism 
\[
\can_{\sO_K}:\pi_*\TCn(\sO_K/\bS[z];\Z_p)[\frac{1}{E}] \to \pi_*\TP(\sO_K/\bS[z];\Z_p)[\frac{1}{E}]
\]
is an isomorphism (see \cite[Proposition 11.10]{BMS2}). This yields an isomorphism
\begin{equation}\label{canisom}
\can_{\sT}[\frac{1}{E}]:\pi_*\TCn(\sT/\bS[z];\Z_p)[\frac{1}{E}] \simeq \pi_*\TP(\sT/\bS[z];\Z_p)[\frac{1}{E}].
\end{equation}
Combine \eqref{canisom} with \eqref{TCnKisin}, we obtain an isomorphism
\[
\pi_i\TCn(\sT/\bS[z];\Z_p) \otimes_{\mS,\varphi} \mS [\frac{1}{E}]\overset{\eqref{TCnKisin}}{\simeq} \pi_i \TP(\sT/\bS[z];\Z_p)[\frac{1}{E}] \overset{\eqref{canisom}}{\simeq} \pi_i\TCn(\sT/\bS[z];\Z_p)[\frac{1}{E}] 
\]
for any $i \geq n$.
\end{proof} 

\subsection{Comparison between $\TCn(\sT/\bS[z];\Z_p)$ and $\TP(\sT_{\sO_\Cu};\Z_p)$} In this section, for an $\sO_K$-linear idempotent-complete, small smooth proper stable $\infty$-category $\sT$, we will compare between $\TCn(\sT/\bS[z];\Z_p)$ and $\TP(\sT_{\sO_\Cu};\Z_p)$. Denote $K_\infty=K(\pi^{1/p^{\infty}})$ and by $\sO_{K_\infty}$ by the integer ring of $K_\infty$.
\begin{lemma}[{\cite[Corollary 11.8]{BMS2}}]\label{BMSTHHlemma}
 For any $\sO_K$-linear stable $\infty$-category $\mathcal{W}$, the natural map
 \[ \THH(\mathcal{W}_{\sO_\Cu};\Z_p) 
 \to \THH(\mathcal{W}_{\sO_\Cu}/\bS[z^{1/p^{\infty}}];\Z_p) 
 \]
 is an equivalence which is compatible with $S^1$-action, and $G_{K_\infty}$-action. In particular, the natural map
 \[ \TP(\mathcal{W}_{\sO_\Cu};\Z_p) 
 \to \TP(\mathcal{W}_{\sO_\Cu}/\bS[z^{1/p^{\infty}}];\Z_p)
 \]
 is an equivalence which is compatible with $G_{K_\infty}$-action.
\end{lemma}
\begin{proof}
The morphism $\bS[z^{1/p^{\infty}}] \to \sO_{K_\infty}~|~ z^{1/p^n}\mapsto \pi^{1/p^n}$ fits into the following commutative diagram
\begin{equation}\label{diagram}
\xymatrix{
 &\sO_K\ar[r] &\sO_{K_\infty} \ar[r] & \sO_{\Cu}\\
\bS\ar[r] &\bS[z]\ar[r] \ar[u]& \bS[z^{1/p^{\infty}}]\ar[u] &
}
\end{equation}
The diagram yields a map 
\[
\THH(\mathcal{W}_{\sO_\Cu};\Z_p) 
 \to \THH(\mathcal{W}_{\sO_\Cu}/\bS[z^{1/p^{\infty}}];\Z_p)\simeq \THH(\mathcal{W}_{\sO_\Cu};\Z_p)\otimes_{\THH(\bS[z^{1/p^{\infty}}];\Z_p)}\bS[z^{1/p^{\infty}}]^{\wedge}_{p}
\]
which is compatible with $S^1$-action. By \cite[Proposition 11.7]{BMS2}, the natural map $\THH(\bS[z^{1/p^{\infty}}];\Z_p) \to \bS[z^{1/p^{\infty}}]^{\wedge}_{p}$ is an equivalence. Since $\pi^{1/p^n}$ is in $K_\infty$ for any $n$, thus the equivalence is $G_{K_\infty}$-equivariant.
\end{proof}

\begin{prop}\label{THHOCperfect}
  For an $\sO_K$-linear idempotent-complete, small smooth proper stable $\infty$-category $\sT$, $\THH(\sT_{\sO_\Cu}/\bS[z^{1/p^{\infty}}];\Z_p)$ is a perfect object in $\Mod_{\THH({\sO_\Cu}/\bS[z^{1/p^{\infty}}];\Z_p)}(\Sp^{BS^1})$. In particular, $\TP(\sT_{\sO_\Cu}/\bS[z^{1/p^{\infty}}];\Z_p)$ is a perfect object in $\Mod_{\TP({\sO_\Cu}/\bS[z^{1/p^{\infty}}];\Z_p)}(\Sp)$.
\end{prop}
\begin{proof}
  We already know $\THH(\sT/\bS[z];\Z_p)$ is a perfect object in $\Mod_{\THH(\sO_K/\bS[z];\Z_p)}(\Sp^{BS^1})$ (see Proposition~\ref{TPTCkunneth}). Since the functor 
  \[
  -\otimes_{\THH(\sO_K/\bS[z];\Z_p)}\THH(\sO_{\Cu}/\bS[z];\Z_p) :\Mod_{\THH(\sO_K/\bS[z];\Z_p)}(\Sp^{BS^1}) \to \Mod_{\THH(\sO_{\Cu}/\bS[z];\Z_p)}(\Sp^{BS^1})
  \]
  is exact and sends $\THH(\sT/\bS[z];\Z_p)$ to $\THH(\sT_{\sO_\Cu}/\bS[z];\Z_p)$, $\THH(\sT_{\sO_\Cu}/\bS[z];\Z_p)$ is a perfect object in $\Mod_{\THH(\sO_{\Cu}/\bS[z];\Z_p)}(\Sp^{BS^1})$. Since the functor 
  \[
  -\otimes_{\THH(\bS[z^{1/p^{\infty}}]/\bS[z];\Z_p)}\bS[z^{1/p^{\infty}}]^{\wedge}_p :\Mod_{\THH(\sO_{\Cu}/\bS[z];\Z_p)}(\Sp^{BS^1}) \to \Mod_{\THH(\sO_{\Cu}/\bS[z^{1/p^\infty}];\Z_p)}(\Sp^{BS^1})
  \]
  is exact and sends $\THH(\sT_{\sO_\Cu}/\bS[z];\Z_p)$ to $\THH(\sT_{\sO_\Cu}/\bS[z^{1/p^{\infty}}];\Z_p)$, thus $\THH(\sT_{\sO_\Cu}/\bS[z^{1/p^{\infty}}];\Z_p)$ is a perfect object in $\Mod_{\THH(\sO_{\Cu}/\bS[z^{1/p^\infty}];\Z_p)}(\Sp^{BS^1})$.
\end{proof}

Lemma~\ref{BMSTHHlemma} and Proposition~\ref{THHOCperfect} imply the following.
\begin{prop}\label{THHOCbSperf}
  For an $\sO_K$-linear idempotent-complete, small smooth proper stable $\infty$-category $\sT$, $\THH(\sT_{\sO_\Cu};\Z_p)$ is a perfect object in $\Mod_{\THH({\sO_\Cu};\Z_p)}(\Sp^{BS^1})$.
\end{prop} 

The following is a non-commutative version of the comparison theorem between $R\Gamma_\mS$ and $R\Gamma_{\Ainf}$ in \cite[Theorem 1.2 (1)]{BMS2}
\begin{thm}\label{TPTCcomp}
Let $\sT$ be an $\sO_K$-linear idempotent-complete, small smooth proper stable $\infty$-category. Then there is a natural number $n$ satisfying that there is a $G_{K_\infty}$-equivariant isomorphism
\[
\pi_i \TCn(\sT/\bS[z];\Z_p)\otimes_{\mS,\ol{\phi}} \Ainf \simeq \pi_i \TP(\sT_{\sO_\Cu}/\bS[z^{1/p^{\infty}}];\Z_p)\simeq \pi_i \TP(\sT_{\sO_\Cu};\Z_p)
\]
for any $i \geq n$, where $\ol{\phi}$ is the map which sends $z$ to $[\pi^{\flat}]^p$ and is the Frobenius on $W$, and $g\in G_{K_\infty}$ acts on $1\otimes g$ on left hand side.
\end{thm}
\begin{proof}
For an $\sO_K$-linear idempotent-complete, small smooth proper stable $\infty$-category $\sT$, consider the following morphism
\[
\TCn(\sT/\bS[z];\Z_p) \overset{\varphi^{hS^1}_{\sT}}{\to} \TP(\sT/\bS[z];\Z_p) \to \TP(\sT_{\sO_\Cu}/\bS[z^{1/p^{\infty}}];\Z_p)
\]
where the second map is given by the diagram~\eqref{diagram}. The map sends $u$ to $\sigma$ and $u$ is an invertible element in $\TP(\sT_{\sO_\Cu}/\bS[z^{1/p^{\infty}}];\Z_p)$, we have a morphism
\begin{equation}
  \underline{\varphi}^{hS^1}_{\sT}:\TCn(\sT/\bS[z];\Z_p)[\frac{1}{u}] \to \TP(\sT_{\sO_\Cu}/\bS[z^{1/p^{\infty}}];\Z_p).
\end{equation}
Let us prove that the map $\underline{\varphi}^{hS^1}_{\sT}$ yields an equivalence of $\TP(\sT_{\sO_\Cu}/\bS[z^{1/p^{\infty}}];\Z_p)$-module:
\begin{equation}\label{AKAMO}
\TCn(\sT/\bS[z];\Z_p)[\frac{1}{u}] \otimes_{\TCn(\sO_K/\bS[z];\Z_p)[\frac{1}{u}],\underline{\varphi}^{hS^1}_{\sO_K}}\TP(\sO_{\Cu}/\bS[z^{1/p^{\infty}}];\Z_p)\simeq\TP(\sT_{\sO_\Cu}/\bS[z^{1/p^{\infty}}];\Z_p)
\end{equation}
which is compatible with $G_{K_\infty}$-action. By Proposition~\ref{TPTCkunneth}, the left side of the map \eqref{AKAMO} yields a symmetric monoidal functor from $\Catsat(\sO_K)$ to $\Mod_{\TP(\sO_{\Cu}/\bS[z^{1/p^{\infty}}];\Z_p)}(\Sp)$. By Proposition~\ref{THHOCperfect}, the right side of the map \eqref{AKAMO} also yields a symmetric monoidal functor from $\Catsat(\sO_K)$ to $\Mod_{\TP(\sO_{\Cu}/\bS[z^{1/p^{\infty}}];\Z_p)}(\Sp)$. By \cite[section 11]{BMS2}, on homotopy groups, the morphism
\[
\underline{\varphi}^{hS^1}_{\sO_K}:\pi_*\TCn(\sO_K/\bS[z];\Z_p)[\frac{1}{u}] \to \pi_*\TP(\sO_{\Cu}/\bS[z^{1/p^\infty}];\Z_p)
\]
is given by $\mS[u^{\pm}] \to \Ainf[\sigma^{\pm}]$ which is $\ol{\phi}$-linear and sends $u$ to $\sigma$, thus it is flat by \cite[Lemma 4.30]{BMS1}. We now have an $G_{K_\infty}$-equivariant isomorphism
\[
\pi_i\TCn(\sT/\bS[z];\Z_p)[\frac{1}{u}] \otimes_{\mS,\ol{\phi}} \Ainf\simeq \pi_i\TP(\sT_{\sO_\Cu}/\bS[z^{1/p^\infty}];\Z_p).
\]
for any $i$. By Lemma \ref{periodlemma} (2), we obtain the claim.
\end{proof}

\subsection{Comparison between $\TCn(\sT/\bS[z];\Z_p)$ and $\TP(\sT_{k};\Z_p)$} In this section, for an $\sO_K$-linear idempotent-complete, small smooth proper stable $\infty$-category $\sT$, we will compare between $\TCn(\sT/\bS[z];\Z_p)$ and $\TP(\sT_{k};\Z_p)$. A Cartesian diagram of $\mathbb{E}_\infty$-ring:
\[
\xymatrix{
k & \sO_K \ar[l] \\
\bS \ar[u] & \bS[z] \ar[u] \ar[l]
}
\]
For an $\sO_K$-linear stable $\infty$-category $\sT$, the diagram yields a morphism 
\[
\TCn(\sT/\bS[z];\Z_p) \to \TCn(\sT_k;\Z_p).
\]
and the composite
\begin{equation}\label{TCnTOkcomp}
\TCn(\sT/\bS[z];\Z_p) \to \TCn(\sT_k;\Z_p) \overset{\varphi^{hS^1}_k}{\to} \TP(\sT_k;\Z_p).
\end{equation}
\begin{thm}\label{TCnTPcomp3}
Let $\sT$ be an $\sO_K$-linear idempotent-complete, small smooth proper stable infinity category. Then there is a natural number $n$ satisfying that the morphism \eqref{TCnTOkcomp} an isomorphism
\[
\pi_i \TCn(\sT/\bS[z];\Z_p)[\frac{1}{p}]\otimes_{\mS [\frac{1}{p}],\tilde{\phi}} K_0 \to \pi_i \TP(\sT_k;\Z_p)[\frac{1}{p}]
\]
for any $i \geq n$, where $\tilde{\phi}$ is the map which sends $z$ to $0$ and Frobenius on $W$.
\end{thm}
\begin{proof}
  The morphism \eqref{TCnTOkcomp} induces a commutative diagram
  \[
  \xymatrix{
  \TCn(\sT/\bS[z];\Z_p) \ar[d]_{\varphi^{hS^1}_{\sT}} \ar[r]& \TCn(\sT_k;\Z_p) \ar[d]^{\varphi^{hS^1}_k} \\
    \TP(\sT/\bS[z];\Z_p)   \ar[r]&\TP(\sT_k;\Z_p) 
  }
  \]
  By Lemma~\ref{periodlemma} (3), it is enough to prove that the morphism $\TP(\sT/\bS[z];\Z_p) \to\TP(\sT_k;\Z_p) $ induces an isomorphism
  \[
  \pi_i \TP(\sT/\bS[z];\Z_p)[\frac{1}{p}]\otimes_{\mS [\frac{1}{p}],\tau} K_0 \to \pi_i \TP(\sT_k;\Z_p)[\frac{1}{p}]
  \]
  for any $i$, where $\tau$ is a $W$-linear and sends $z$ to $0$. By Theorem~\ref{BKtheorem} and \cite[Proposition 4.3]{BMS1}, there is $n$ such that for any $i\geq n$, the homotopy group $\pi_i\TCn(\sT/\bS[z];\Z_p)$ is a Breuil-Kisin module, and $\pi_i\TCn(\sT/\bS[z];\Z_p)[\frac{1}{p}]$ is a finite free $\mS[\frac{1}{p}]$-module. By an isomorphism~\ref{TCnKisin} and the fact that $\TP(\sT/\bS[z];\Z_p)$ is $2$-periodic, we obtain that $\pi_i\TP(\sT/\bS[z];\Z_p)[\frac{1}{p}]\overset{\eqref{TCnKisin}}{\simeq} \pi_i\TCn(\sT/\bS[z];\Z_p)[\frac{1}{p}]\otimes_{\mS[\frac{1}{p}],\phi}\mS[\frac{1}{p}]$ is a finite free $\mS[\frac{1}{p}]$-module for any $i$. Thus, on homotopy groups, one obtains an isomorphism 
  \begin{eqnarray}\label{flatmap}
 \pi_*\TP(\sT/\bS[z];\Z_p)[\frac{1}{p}]\simeq \pi_0\TP(\sT/\bS[z];\Z_p)[\frac{1}{p}]\otimes_{\mS[\frac{1}{p}]} \pi_*\TP(\sO_K/\bS[z];\Z_p)[\frac{1}{p}] \\
 \oplus \pi_1\TP(\sT/\bS[z];\Z_p)[\frac{1}{p}]\otimes_{\mS[\frac{1}{p}]}\pi_*\TP(\sO_K/\bS[z];\Z_p)[\frac{1}{p}],\nonumber
  \end{eqnarray}
  thus $\pi_*\TP(\sT/\bS[z];\Z_p)[\frac{1}{p}]$ is a flat graded $\pi_*\TP(\sO_K/\bS[z];\Z_p)[\frac{1}{p}]$-module. Let us prove the morphism
  \begin{equation}\label{ACOMU}
  \TP(\sT/\bS[z];\Z_p)[\frac{1}{p}] \otimes_{\TP(\sO_K/\bS[z];\Z_p)[\frac{1}{p}]} \TP(k;\Z_p)[\frac{1}{p}] \to \TP(\sT_k;\Z_p)[\frac{1}{p}]
  \end{equation}
  is an equivalence. Both sides of the map \eqref{ACOMU} yield symmetric monoidal functors from $\Catsat(\sO_K)$ to $\Mod_{\TP(k;\Z_p)[\frac{1}{E}]}(\Sp)$, thus by Proposition~\ref{compmonoid}, we know the map \eqref{ACOMU} is an equivalence. By the isomorphism \eqref{flatmap} and the fact that $\TP(\sT_{\sO_\Cu};\Z_p)$ is $2$-periodic, on homotopy groups, we have an isomorphism
\begin{equation}\label{TPTPcomp}
\pi_i \TP(\sT/\bS[z];\Z_p)[\frac{1}{p}]\otimes_{\mS [\frac{1}{p}],\tau} K_0 \to \pi_i \TP(\sT_k;\Z_p)[\frac{1}{p}]
\end{equation}
for any $i$. By isomorphism~\eqref{TCnKisin} and an equality $\tau\circ \phi=\tilde{\phi}$, there is an isomorphism
\begin{eqnarray*}
\pi_i \TCn(\sT/\bS[z];\Z_p)[\frac{1}{p}]\otimes_{\mS [\frac{1}{p}],\tilde{\phi}} K_0 & \simeq & \pi_i \TCn(\sT/\bS[z];\Z_p)[\frac{1}{p}]\otimes_{\mS [\frac{1}{p}],\phi} \mS \otimes_{\mS,\tau} K_0 \\
&\simeq & \pi_i \TP(\sT/\bS[z];\Z_p)[\frac{1}{p}]\otimes_{\mS [\frac{1}{p}],\tau} K_0 \\
&\overset{\eqref{TPTPcomp}}{\simeq} &\pi_i \TP(\sT_k;\Z_p)[\frac{1}{p}].
\end{eqnarray*}
for any $i\geq n$.
\end{proof}

\subsection{Comparison theorem between $\TCn(\sT/\bS[z];\Z_p)$ and $\LK K(\sT_{\Cu})$} Let $\sT$ be an $\sO_K$-linear idempotent-complete, small smooth proper stable infinity category. Using \cite[Thm 2.16]{remarkonK1}, one obtain an equivalence $\LK K(\sT_{\sO_{\Cu}}) \simeq \LK K(\sT_\Cu)$ of ring spectra. We recall some results about topological Hochschild homology theory from \cite{Lars}, \cite{remarkonK1} and \cite{BMS2}. On homotopy groups, there is an isomorphism
\begin{equation}
  \pi_*\TP(\sO_\Cu;\Z_p) \simeq \Ainf[\sigma,\sigma^{-1}]
\end{equation}
where $\sigma$ is a generator $\sigma \in \TP_2(\sO_\Cu;\Z_p)$. Let $\beta\in K_2(\Cu)$ be the Bott element. For any $\sO_\Cu$-linear stable $\infty$-category $\sD$, we have an identification $\LK K(\sD) \simeq \LK K(\sD_\Cu)$ (see \cite[Theorem~2.16]{remarkonK1}). The cyclotomic trace map $\text{tr}:K(\sO_\Cu) \to \TP(\sO_\Cu;\Z_p)$ sends $\beta$ to a $\Z_p^*$-multiple of $([\epsilon]-1)\sigma$ (see \cite{Lars} and \cite[Theorem 1.3.6]{HN}), the cyclotomic trace map induce a morphism
\[
\LK\text{tr}:\LK K(\sD_\Cu)\simeq \LK K(\sD) \to \LK \TP(\sD)\simeq \TP(\sD;\Z_p)[\frac{1}{[\epsilon]-1}]^{\wedge}_p
\]
for any $\sO_\Cu$-linear stable $\infty$-category $\sD$.
\begin{thm}\label{KTCncomp}
Let $\sT$ be an $\sO_K$-linear idempotent-complete, small smooth proper stable infinity category. We assume Conjecture \ref{kunnethK1}, then the trace map induces an equivalence
\begin{equation}\label{KTCnequi}
\LK K(\sT_\Cu)\otimes_{\LK K(\Cu)} \LK \TP(\sO_\Cu) \to \LK \TP(\sT_{\sO_\Cu}).
\end{equation}
In particular, there is a $G_K$-equivariant isomorphism
\[
\pi_i \LK K(\sT_\Cu) \otimes_{\Z_p} \Ainf[\frac{1}{[\epsilon]-1}]^{\wedge}_p \simeq \pi_i \TP(\sT_{\sO_\Cu};\Z_p)[\frac{1}{[\epsilon]-1}]^{\wedge}_p
\]
for any $i$.
\end{thm}
\begin{proof}
Both sides of the map \eqref{KTCnequi} yield symmetric monoidal functors from $\Catsat(\sO_K)$ to $\Mod_{\LK\TP(\sO_\Cu;\Z_p)}(\Sp)$, thus the map \eqref{KTCnequi} is an equivalence by the assumption of Conjecture~\ref{kunnethK1} and Proposition~\ref{THHOCbSperf}. On homotopy groups, the morphism $\LK \tr :\pi_*\LK K(\Cu)\to \pi_*\LK\TP(\sO_\Cu;\Z_p)$ is given by $\Z_p[\beta^{\pm}] \to \Ainf[\frac{1}{[\epsilon]-1}]^{\wedge}_p[\sigma^{\pm}]$ which sends $\beta$ to $([\epsilon]-1)\sigma$, and is flat. Thus we obtain the claim.
\end{proof}

The following corollary is a non-commutative analogue of the finiteness of \'etale cohomology groups for proper varieties over a closed field.
\begin{cor}
Let $\sT$ be an $\sO_K$-linear idempotent-complete, small smooth proper stable infinity category. We assume Conjecture \ref{kunnethK1}, then all $\pi_*\LK K(\sT_\Cu)$ is a finitely generated $\Z_p$-module.
\end{cor}
\begin{proof}
Note that $\Z_p \to \Ainf[\frac{1}{[\epsilon]-1}]$ is faithful flat. By Theorem~\ref{KTCncomp}, Lemma~\ref{BMSTHHlemma} and Proposition~\ref{THHOCperfect} we obtain the claim. 
\end{proof}

Without the assumption of Conjecture~\ref{kunnethK1}, we can prove the following.
\begin{thm}\label{KTPcompcommutative}
  Assume $K$ is a finite extension of $\Q_p$. Let $X$ be a smooth proper variety over $K$. Then there is a $G_K$-equivariant isomorphism
\[
\pi_i \LK K(X_\Cu) \otimes_{\Z_p} \Ainf[\frac{1}{[\epsilon]-1}]^{\wedge}_p \simeq \pi_i \TP(X_{\sO_\Cu};\Z_p)[\frac{1}{[\epsilon]-1}]^{\wedge}_p
\]
for any $i$.
\end{thm}
\begin{proof}
  The natural map
  \[
  \LK K(X_\Cu)\otimes_{\LK K(\Cu)} \LK \TP(\sO_\Cu) \to \LK \TP(X_{\sO_\Cu})
  \]
  is an equivalence, since both sides of the map yield symmetric monoidal functors from $\Catcsat(\sO_K)$ to $\Mod_{\LK\TP(\sO_\Cu;\Z_p)}(\Sp)$ by the assumption of Proposition~\ref{commkunnethK1} and Proposition~\ref{THHOCbSperf}. On homotopy groups, the morphism $\LK \tr :\pi_*\LK K(\Cu)\to \pi_*\LK\TP(\sO_\Cu;\Z_p)$ is given by $\Z_p[\beta^{\pm}] \to \Ainf[\frac{1}{[\epsilon]-1}]^{\wedge}_p[\sigma^{\pm}]$ which sends $\beta$ to $([\epsilon]-1)\sigma$, and is flat. Thus we obtain the claim.
\end{proof}

\subsection{Non-commutative version of Bhatt-Morrow-Scholze's comparison theorem} In this section, we prove non-commutative version of Theorem~\ref{BKcoh}.
\begin{thm}\label{NCBMS}
  Let $\sT$ be an $\sO_K$-linear idempotent-complete, small smooth proper stable $\infty$-category. Then there is $n$ satisfying the following holds:
 \begin{itemize}
   \item[(1)] For any $i\geq n$, $\pi_i \TCn(\sT/\bS[z];\Z_p)$ is a Breuil-Kisin module.
   \item[(2)] (K(1)-local K theory comparison) Assume Conjecture \ref{kunnethK1}. For any $i\geq n$, after scalar extension along $\ol{\phi}:\mS\to \Ainf$, one recovers K(1)-local K theory of generic fiber
   \[
   \pi_i \TCn(\sT/\bS[z];\Z_p) \otimes_{\mS,\ol{\phi}} \Ainf[\frac{1}{\mu}]^{\wedge}_p \simeq \pi_i\LK K(\sT_{\Cu})\otimes_{\Z_p} \Ainf[\frac{1}{\mu}]^{\wedge}_p
   \]
   \item[(3)] (topological periodic homology theory comparison) For any $i\geq n$, after scalar extension along the map $\tilde{\phi}:\mS\to W$ which is the Frobenius on $W$ and sends $z$ to $0$, one recovers topological periodic homology theory of the special fiber:
  \[
   \pi_i \TCn(\sT/\bS[z];\Z_p)[\frac{1}{p}] \otimes_{\mS[\frac{1}{p}],\tilde{\phi}} K_0 \simeq \pi_i \TP(\sT_k;\Z_p)[\frac{1}{p}]
   \]
    \item[(4)] If there is a smooth proper variety $X$ such that $\sT=\perf(X)$, then for amy $i\geq n$, there is an isomorphism
   \[
    \pi_i \TCn(X/\bS[z];\Z_p) \otimes_{\mS,\ol{\phi}} \Ainf[\frac{1}{\mu}]^{\wedge}_p \simeq \pi_i\LK K(X_{\Cu})\otimes_{\Z_p} \Ainf[\frac{1}{\mu}]^{\wedge}_p
   \]
   \end{itemize}
\end{thm}
\begin{proof}
  Claim (1) follows from Theorem~\ref{BKtheorem}. Claim (2) follows from Theorem~\ref{TPTCcomp} and Theorem~\ref{KTCncomp}. Claim (3) follows from Theorem~\ref{TCnTPcomp3}. Claim (4) follows from Theorem~\ref{KTPcompcommutative}.
\end{proof}
\begin{cor}\label{Bcrycor}
  We assume the conjecture~\ref{kunnethK1} holds. For an $\sO_K$-linear idempotent-complete, small smooth proper stable $\infty$-category $\sT$ and an integer $i$, there is an isomorphism of $\Bcry$-modules:
\[
\pi_i\TP(\sT_k;\Z_p)\otimes_W \Bcry \simeq \pi_i L_{K(1)}K(\sT_\Cu)\otimes_{\Z_p} \Bcry.
\]
\end{cor}
\begin{proof}
  The claim follows from Theorem~\ref{NCBMS} amd Lemma~\ref{BKlemma}.
\end{proof}
\section{Galois representation of K(1)-local K theory}

\subsection{Breuil-Kisin $G_K$-module} Let us recall the work of Gao on Breuil-Kisin $G_K$-module \cite{Gao}. Consider the following diagram: 
\[
\xymatrix{
&& \Ainf \\
\Ainf \ar@/^18pt/[rru]^{\varphi_{\Ainf}} \ar@{}[rd]|{\square} \ar[r]^{\varphi\otimes_{\mS} 1_{\Ainf}} & \Ainf \ar[ru]^{i}& \\
  \mS \ar[u]^{\phi} \ar[r]_{\varphi} & \mS \ar[u]_{\phi}\ar@/_18pt/[ruu]_{\phi}&
}
\]
Let us denote by $\xi$ a generator of the kernel of the map $\Ainf \to \sO_\Cu$. Note that $\phi(E)=\xi$. Let $(\fM,\varphi_\fM)$ be a Breuil-Kisin module. We denote $\widehat{\fM}=\fM\otimes_{\mS,\phi}\Ainf$. After tensoring $-\otimes_{\mS,\phi} \Ainf$, Frobenius $\fM\otimes_{\mS,\varphi}\mS[\frac{1}{E}]\overset{\varphi_{\fM}}{\simeq}\fM[\frac{1}{E}]$ becomes $\widehat{\fM}\otimes_{\Ainf,\varphi\otimes 1_{\Ainf}} \Ainf[\frac{1}{\xi}] \simeq \widehat{\fM}[\frac{1}{\xi}]$, and after tensoring $-\otimes_{\Ainf,i}\Ainf$ we obtain an isomorphism
\begin{equation}\label{power}
\widehat{\varphi}_{\fM}:\widehat{\fM}\otimes_{\Ainf,\varphi_{\Ainf}} \Ainf[\frac{1}{\xi}] \simeq \widehat{\fM}[\frac{1}{\xi}].
\end{equation}

\begin{defn}[{\cite[Defn 7.1.1]{Gao}}]
Let $(\fM,\varphi_\fM)$ be a finite free Breuil-Kisin module. We call $(\fM,\varphi_\fM)$ Breuil-Kisin $G_K$-module if it satisfies the following conditions.
\begin{itemize}
  \item[(1)] There is a continuous $\Ainf$-semi-linear $G_K$-action on $\widehat{\fM}=\fM\otimes_{\mS,\phi}\Ainf$.
  \item[(2)] $G_K$ commutes with $\widehat{\varphi}_\fM$.
  \item[(3)] $\fM \subset \widehat{\fM}^{G_{K_\infty}}$ via the embedding $\fM \subset \widehat{\fM}$.
  \item[(4)] $\fM/z\fM\subset \bigl(\widehat{\fM}\otimes_{\Ainf}W(\ol{k})\bigr)^{G_K}$ via the embedding $\fM/z\fM\subset \widehat{\fM}\otimes_{\Ainf}W(\ol{k})$.
\end{itemize}
\end{defn}

The natural surjective map of rings $\sO_\Cu \to \ol{k}$ induces a morphism of ring spectra
\begin{equation}\label{TPTPOColkmap}
\TP(\sT_{\sO_\Cu};\Z_p) \to \TP(\sT_{\ol{k}};\Z_p).
\end{equation}
By an equivalence $\THH(\sT_{\ol{k}};\Z_p)\simeq \THH(\ol{k};\Z_p) \otimes_{\THH(\sO_\Cu;\Z_p)} \THH(\sT_{\sO_\Cu};\Z_p)$ and Proposition~\ref{THHOCbSperf}, we have the following.
\begin{prop}\label{THHolkperf}
For an $\sO_K$-linear idempotent-complete, small smooth proper stable infinity category $\sT$, $\THH(\sT_{\ol{k}};\Z_p)$ is a perfect object in $\Mod_{\THH({\ol{k}};\Z_p)}(\Sp^{BS^1})$.
\end{prop}

Let us recall calculations of $\TP$ of $\sO_\Cu$ and $\ol{k}$ from \cite{BMS2}. On homotopy groups,
\begin{equation}
\varphi^{hS^1}_{\sO_\Cu}:  \pi_* \TCn(\sO_\Cu;\Z_p) \to \pi_*\TP(\sO_\Cu;\Z_p)
\end{equation}
is given by 
\[
\Ainf[u,v]/(uv-\xi) \to \Ainf[\sigma^\pm]
\]
which is $\varphi_{\Ainf}$-linear and sends $u$ to $\sigma$ and $v$ to $\varphi_{\Ainf}(\xi)\sigma^{-1}$. Similarly, on homotopy groups,
\begin{equation}
\varphi^{hS^1}_{\ol{k}}: \pi_* \TCn(\ol{k};\Z_p) \to \pi_*\TP(\ol{k};\Z_p)
\end{equation}
is given by 
\[
W(\ol{k})[u,v]/(uv-p) \to W(\ol{k})[\sigma^\pm]
\]
which is $\varphi_{W(\ol{k})}$-linear and sends $u$ to $\sigma$ and $v$ to $p\sigma^{-1}$.

\begin{thm}\label{TPTPOColkcomp}
Let $\sT$ be an $\sO_K$-linear idempotent-complete, small smooth proper stable infinity category. The morphism~\eqref{TPTPOColkmap} induces an isomorphism
\[
\pi_i\TP(\sT_{\sO_\Cu};\Z_p)[\frac{1}{p}] \otimes_{\Ainf[\frac{1}{p}]} W(\ol{k})[\frac{1}{p}] \simeq \pi_i\TP(\sT_{\ol{k}};\Z_p)[\frac{1}{p}]
\]
for any $i$. In particular, $\pi_i\TP(\sT_{\sO_\Cu};\Z_p)[\frac{1}{p}]\otimes_{\Ainf[\frac{1}{p}]} W(\ol{k})[\frac{1}{p}]$ is fixed by $G_{K^{ur}}$.
\end{thm}
\begin{proof}
The morphism~\eqref{TPTPOColkmap} yields a morphism
\begin{equation}\label{TPTPOColkmap2}
\TP(\sT_{\sO_\Cu};\Z_p)[\frac{1}{p}]\otimes_{\TP(\sO_\Cu;\Z_p)[\frac{1}{p}]} \TP(\ol{k};\Z_p)[\frac{1}{p}] \to \TP(\sT_{\ol{k}};\Z_p)[\frac{1}{p}].
\end{equation}
 By Proposition~\ref{THHOCbSperf} and Proposition~\ref{THHolkperf}, both sides of the map \eqref{TPTPOColkmap} yield symmetric monoidal functors from $\Catsat(\sO_K)$ to $\Mod_{\TP(\ol{k};\Z_p)}(\Sp)$, and the map \eqref{TPTPOColkmap2} yield a symmetric monoidal natural transformation between them. Thus the morphism \eqref{TPTPOColkmap2} is an equivalence. By Theorem~\ref{BKtheorem} and \cite[Proposition 4.3]{BMS1}, there is $n$ such that for any $i\geq n$, the homotopy group $\pi_i\TCn(\sT/\bS[z];\Z_p)$ is a Breuil-Kisin module, and $\pi_i\TCn(\sT/\bS[z];\Z_p)[\frac{1}{p}]$ is a finite free $\mS[\frac{1}{p}]$-module. By Theorem~\ref{TPTCcomp} and the fact that $\TP(\sT_{\sO_\Cu};\Z_p)$ is $2$-periodic, $\pi_i\TP(\sT_{\sO_\Cu};\Z_p)[\frac{1}{p}]$ is a finite free $\Ainf[\frac{1}{p}]$-module for any $i$. Thus we have an isomorphism of graded $\pi_*\TP(\sO_\Cu;\Z_p)[\frac{1}{p}]$-modules \begin{eqnarray}\label{flatmap2}
 \pi_*\TP(\sT_{\sO_\Cu};\Z_p)[\frac{1}{p}]\simeq \pi_0\TP(\sT_{\sO_\Cu};\Z_p)[\frac{1}{p}]\otimes_{\Ainf[\frac{1}{p}]} \pi_*\TP(\sO_\Cu;\Z_p)[\frac{1}{p}] \\
 \oplus \pi_1\TP(\sT_{\sO_\Cu};\Z_p)[\frac{1}{p}]\otimes_{\Ainf[\frac{1}{p}]}\pi_*\TP(\sO_\Cu;\Z_p)[\frac{1}{p}],\nonumber
  \end{eqnarray}
  thus $\pi_*\TP(\sT_{\sO_\Cu};\Z_p)[\frac{1}{p}]$ is a flat graded $\pi_*\TP(\sO_\Cu;\Z_p)[\frac{1}{p}]$-module. We obtain the claim.
\end{proof}

\begin{lemma}\label{periodlemma2}
Let $\sT$ be an $\sO_K$-linear idempotent-complete, small smooth proper stable infinity category. Then the followings hold:
\begin{itemize}
  \item[(1)] $\pi_*\TCn(\sT_{\sO_\Cu};\Z_p)$ is a finitely generated $\pi_*\TCn(\sO_\Cu/\bS[z];\Z_p)$-module, thus all $\pi_i\TCn(\sT_{\sO_\Cu};\Z_p)$ is finitely generated $\Ainf$-module.
  \item[(2)] There is a natural number $n$ satisfying that for any $j \geq n$, $\pi_j\TCn(\sT_{\sO_\Cu};\Z_p)\overset{u\cdot }{\to} \pi_{j+2}\TCn(\sT_{\sO_\Cu};\Z_p)$ is an isomorphism. 
\end{itemize}
\end{lemma}
\begin{proof}
Claim (1) directly follows from Proposition ~\ref{THHOCbSperf}. For any $j \geq 0$, by the calculation of $\pi_* \TCn({\sO_\Cu};\Z_p)$ (see \cite[Proposition 11.10]{BMS2}), we know $\pi_j\TCn({\sO_\Cu};\Z_p)\overset{u\cdot}{\to} \pi_{j+2}\TCn({\sO_\Cu};\Z_p)$ is an isomorphism. These yields claim (2) by claim (1).
\end{proof}

\begin{thm}\label{TPTCnOCOCcomp}
Let $\sT$ be an $\sO_K$-linear idempotent-complete, small smooth proper stable infinity category. There is a $n$ such that the cyclotomic Frobenius morphism $\varphi^{hS^1}_{\sT_{\sO_\Cu}}:\TCn(\sT_{\sO_\Cu};\Z_p) \to \TP(\sT_{\sO_\Cu};\Z_p)$ induces a $G_K$-equivariant isomorphism
\begin{equation}\label{TCnTPOCAinfequi}
\pi_i\TCn(\sT_{\sO_\Cu};\Z_p)\otimes_{\Ainf,\varphi_{\Ainf}} \Ainf \simeq \pi_i\TP(\sT_{\sO_\Cu};\Z_p)
\end{equation}
for any $i\geq n$. 
\end{thm}
\begin{proof} Cyclotomic Frobenius map $\varphi^{hS^1}_{\sT_{\sO_\Cu}}:\TCn(\sT_{\sO_\Cu};\Z_p) \to \TP(\sT_{\sO_\Cu};\Z_p)$ yields a morphism 
\[
\underline{\varphi}^{hS^1}_{\sT_{\sO_\Cu}}:\TCn(\sT_{\sO_\Cu};\Z_p)[\frac{1}{u}] \to \TP(\sT_{\sO_\Cu};\Z_p).
\]
The morphism $\underline{\varphi}^{hS^1}_{\sT_{\sO_\Cu}}$ induces a morphism of $\TP(\sO_\Cu;\Z_p)$-module: 
  \begin{equation}\label{TCnTPOCequi}
  \TCn(\sT_{\sO_\Cu};\Z_p)[\frac{1}{u}]\otimes_{\TCn(\sO_\Cu;\Z_p)[\frac{1}{u}],\underline{\varphi}^{hS^1}_{\sO_\Cu}} \TP(\sO_\Cu;\Z_p)\to \TP(\sT_{\sO_\Cu};\Z_p).
  \end{equation}
  Both sides of the map \eqref{KTCnequi} yield symmetric monoidal functors from $\Catsat(\sO_K)$ to $\Mod_{\TP(\sO_\Cu;\Z_p)}(\Sp)$, thus the map \eqref{KTCnequi} is an equivalence by Proposition~\ref{THHOCbSperf}. On homotopy groups, $\underline{\varphi}^{hS^1}_{\sO_\Cu}: \pi_*\TCn(\sO_\Cu;\Z_p)[\frac{1}{u}] \to \pi_* \TP(\sO_\Cu;\Z_p)$ is given by $\varphi_{\Ainf}$-linear map
  \[
  \Ainf[u^\pm] \to \Ainf[\sigma^\pm],
  \]
  thus is an isomorphism. Thus we obtain the isomorphism
  \[
  \pi_i \TCn(\sT_{\sO_\Cu};\Z_p)[\frac{1}{u}] \otimes_{\Ainf,\varphi_{\Ainf}} \pi_i \TP(\sT_{\sO_\Cu};\Z_p)
  \]
  for any $i$. By Lemma~\ref{periodlemma2}, there is a natural number $n$ such that $\pi_i \TCn(\sT_{\sO_\Cu};\Z_p)\simeq \pi_i \TCn(\sT_{\sO_\Cu};\Z_p)[\frac{1}{u}]$ for any $i \geq n$.
\end{proof}

For a Breuil-Kisin module $\fM$, let us denote $\fM^{\vee}=\Hom_{\mS}(\fM,\mS)$. We note $\fM^{\vee}=\Hom_{\mS}(\fM,\mS)$ is a Breuil-Kisin module whose Frobenius map $\varphi_{\fM^{\vee}}$ is given by 
\begin{equation}\label{dualFrob}
  \fM^\vee \otimes_{\mS,\phi} \mS[\frac{1}{E}] \simeq \Hom_{\mS[\frac{1}{E}]}(\fM\otimes_{\mS,\phi} \mS[\frac{1}{E}],\mS[\frac{1}{E}]) \overset{\phi_{\fM}^\vee}{\simeq} \Hom_{\mS[\frac{1}{E}]}(\fM[\frac{1}{E}],\mS[\frac{1}{E}]) \simeq \fM^\vee[\frac{1}{E}],
\end{equation}
where we use facts that $\phi:\mS\to\mS$ is flat and $\mS\to \mS[\frac{1}{E}]$ is flat. Since $\phi:\mS\to \Ainf$ is flat, $\widehat{\fM^{\vee}} \simeq (\widehat{\fM})^{\vee}:= \Hom_{\Ainf}(\widehat{\fM},\Ainf)$. we have an isomorphism
\begin{eqnarray}\label{dualAinfFrob}
(\widehat{\fM})^{\vee}\otimes_{\Ainf,\varphi_{\Ainf}}\Ainf[\frac{1}{\xi}]\simeq \Hom_{\Ainf[\frac{1}{\xi}]}(\widehat{\fM}\otimes_{\Ainf,\varphi_{\Ainf}}\Ainf[\frac{1}{\xi}], \Ainf[\frac{1}{\xi}])\\
\overset{\widehat{\varphi_{{\fM}}}^{\vee}}{\simeq} \Hom_{\Ainf[\frac{1}{\xi}]}(\widehat{\fM}[\frac{1}{\xi}], \Ainf[\frac{1}{\xi}]) \simeq (\widehat{\fM})^{\vee}[\frac{1}{\xi}], \nonumber
\end{eqnarray}
and it is the same as $\widehat{\varphi}_{\fM^\vee}$ since the morphism~\eqref{dualAinfFrob} is given by the tensoring $-\otimes_{\mS,\phi}\Ainf$ of \eqref{dualFrob}. We assume that there is a continuous $\Ainf$-semi-linear $G_K$-action on $\widehat{\fM}=\fM\otimes_{\mS,\phi}\Ainf$ and $G_K$ commutes with $\widehat{\varphi}_\fM$. All constructions in \eqref{dualAinfFrob} are functorial, and thus there is a continuous $\Ainf$-semi-linear $G_K$-action on $\widehat{\fM}^{\vee}=\fM^{\vee}\otimes_{\mS,\phi}\Ainf$ and $G_K$ commutes with $\widehat{\varphi}_{\fM^\vee}$.

For an $\sO_K$-linear idempotent-complete, small smooth proper stable $\infty$-category $\sT$, we will prove that the free Breuil-Kisin module $\pi_i \TCn(\sT/\bS[z];\Z_p)_{\free}:=\pi_i \TCn(\sT/\bS[z];\Z_p)^{\vee\vee}$ is a Breuil-Kisin $G_K$-module. 

\begin{thm}\label{BKGKmodulethm}
  Let $\sT$ be an $\sO_K$-linear idempotent-complete, small smooth proper stable infinity category. There is a $n$ such that for any $i\geq n$, the followings holds:
  \begin{itemize}
    \item[(1)] There is a continuous $\Ainf$-semi-linear $G_K$-action on $\widehat{\pi_i\TCn(\sT/\bS[z];\Z_p)_{\free}}=\pi_i\TCn(\sT/\bS[z];\Z_p)_{\free}\otimes_{\mS,\phi}\Ainf$
    \item[(2)] $G_K$ commutes with $\widehat{\varphi}_{\pi_i \TCn(\sT/\bS[z];\Z_p)_{\free}}$.
    \item[(3)] $\pi_i \TCn(\sT/\bS[z];\Z_p)_{\free} \subset \bigl(\widehat{\pi_i \TCn(\sT/\bS[z];\Z_p)_{\free}}\bigr)^{G_{K_\infty}}$ via the embedding $\pi_i \TCn(\sT/\bS[z];\Z_p)_{\free} \subset \widehat{\pi_i \TCn(\sT/\bS[z];\Z_p)_{\free}}$. 
    \item[(4)] $\pi_i\TCn(\sT_{\sO_\Cu};\Z_p)_{\free}[\frac{1}{p}]\otimes_{\Ainf[\frac{1}{p}]} W(\ol{k})[\frac{1}{p}]$ is fixed by $G_{K^{ur}}$.
  \end{itemize}
  In particular, $\pi_i \TCn(\sT/\bS[z];\Z_p)$ is a Breuil-Kisin $G_K$-module.
\end{thm}
\begin{proof}
After localization $\xi\in \Ainf \simeq \pi_0 \TCn({\sO_{\Cu}};\Z_p)$, the morphism $\can_\sT$ induces a morphism of $\TP(\sO_\Cu;\Z_p)[\frac{1}{\xi}]$-module:
\begin{equation}\label{TCTPOCcan}
\TCn(\sT_{\sO_{\Cu}};\Z_p)[\frac{1}{\xi}] \otimes_{\TCn({\sO_{\Cu}};\Z_p)[\frac{1}{\xi}],\can_{{\sO_{\Cu}}}} \TP({\sO_{\Cu}};\Z_p)[\frac{1}{\xi}] \to \TP(\sT_{\sO_{\Cu}};\Z_p)[\frac{1}{\xi}].
\end{equation}
Both sides of the map \eqref{TCTPOCcan} yield symmetric monoidal functors from $\Catsat(\sO_K)$ to $\Mod_{\TP(\sO_{\sO_{\Cu}};\Z_p)[\frac{1}{\xi}]}(\Sp)$, and the map \eqref{TCTPOCcan} yield a symmetric monoidal natural transformation between them. Thus the morphism \eqref{TCTPOCcan} is an equivalence. Note that the morphism 
\[
\can_{{\sO_{\Cu}}}:\pi_*\TCn({\sO_{\Cu}};\Z_p)[\frac{1}{\xi}] \to \pi_*\TP({\sO_{\Cu}};\Z_p)[\frac{1}{\xi}]
\]
is an isomorphism (see \cite[Proposition 11.10]{BMS2}). This yields a $G_K$-equivariant isomorphism
\begin{equation}\label{canisom2}
\can_{\sT_{\sO_{\Cu}}}[\frac{1}{\xi}]:\pi_*\TCn(\sT_{\sO_{\Cu}};\Z_p)[\frac{1}{\xi}] \simeq \pi_*\TP(\sT_{\sO_{\Cu}};\Z_p)[\frac{1}{\xi}].
\end{equation}
Combine \eqref{TCnTPOCAinfequi} with \eqref{canisom2}, we obtain a $G_K$-equivariant isomorphism
\begin{equation}\label{FrobTCnOC}
\pi_i\TCn(\sT_{\sO_\Cu};\Z_p)\otimes_{\Ainf,\varphi_{\Ainf}} \Ainf[\frac{1}{\xi}]\overset{\eqref{TCnTPOCAinfequi}}{\simeq} \pi_i\TP(\sT_{\sO_\Cu};\Z_p)[\frac{1}{\xi}] \overset{\eqref{canisom2}}{\simeq}\pi_i\TCn(\sT_{\sO_{\Cu}};\Z_p)[\frac{1}{\xi}]
\end{equation}
for any $i \geq n$. Since $\varphi_{\Ainf}$ is an isomorphism, by Theorem~\ref{TPTCcomp} and Theorem~\ref{TPTCnOCOCcomp}, we have a $G_{K_\infty}$-equivariant isomorphism
\begin{equation}\label{BIWAKOUP}
\pi_i\TCn(\sT/\bS[z];\Z_p)\otimes_{\mS,\phi}\Ainf\simeq \pi_i\TCn(\sT_{\sO_\Cu};\Z_p),
\end{equation}
 and the isomorphism \eqref{FrobTCnOC} is same as $\widehat{\varphi}_{\pi_i\TCn(\sT/\bS[z];\Z_p)}$. Since $G_K$ naturally acts on $\pi_i\TCn(\sT_{\sO_\Cu};\Z_p)$, we obtain the claim (1). By the isomorphism~\eqref{FrobTCnOC}, $G_K$ commutes with $\widehat{\varphi}_{\pi_i \TCn(\sT/\bS[z];\Z_p)}$ we obtain the claim (2). 
 
 Since $\varphi:\mS\to\Ainf$ is flat, we have a $G_{K_\infty}$-equivariant
 \begin{equation}\label{kenka1} \pi_i\TCn(\sT/\bS[z];\Z_p)_{\free}\otimes_{\mS,\phi}\Ainf\simeq \pi_i\TCn(\sT_{\sO_\Cu};\Z_p)_{\free}  
 \end{equation}
 by \eqref{BIWAKOUP}. Since $g\in G_{K_\infty}$ acts on $1\otimes g$ on $\pi_i\TCn(\sT/\bS[z];\Z_p)_{\free}\otimes_{\mS,\phi}\Ainf$, we obtain the claim (3).
 
Since $\varphi_{\Ainf}:\Ainf \to \Ainf$ is flat, we have $G_K$-equivariant isomorphism
\begin{equation}\label{freeTCTP}
\pi_i\TCn(\sT_{\sO_\Cu};\Z_p)_{\free}\otimes_{\Ainf,\varphi_{\Ainf}} \Ainf \simeq \pi_i\TP(\sT_{\sO_\Cu};\Z_p)_{\free}
\end{equation}
by Theorem~\ref{TPTCnOCOCcomp}. By Theorem~\ref{TPTCcomp} and \cite[Proposition 4.13]{BMS1}, the natural map $\pi_i\TP(\sT_{\sO_\Cu};\Z_p)[\frac{1}{p}] \to \pi_i\TP(\sT_{\sO_\Cu};\Z_p)_{\free}[\frac{1}{p}]$ is a $G_K$-equivariant isomorphism. By Theorem~\ref{TPTPOColkcomp} and an isomorphism \eqref{freeTCTP}, $\pi_i\TCn(\sT_{\sO_\Cu};\Z_p)_{\free}[\frac{1}{p}] \otimes_{\Ainf[\frac{1}{p}]} W(\ol{k})[\frac{1}{p}]$ is fixed by $G_{K^{ur}}$.

Since $\pi_i\TCn(\sT_{\sO_\Cu};\Z_p)_{\free}$ is finite free $\mS$-module, there is a natural inclusion $\pi_i\TCn(\sT_{\sO_\Cu};\Z_p)_{\free} \hookrightarrow\pi_i\TCn(\sT_{\sO_\Cu};\Z_p)_{\free}[\frac{1}{p}]$ and it induces a natural inclusion
\[
\pi_i\TCn(\sT_{\sO_\Cu};\Z_p)_{\free} \otimes_{\Ainf} W(\ol{k}) \hookrightarrow \pi_i\TCn(\sT_{\sO_\Cu};\Z_p)_{\free}[\frac{1}{p}] \otimes_{\Ainf[\frac{1}{p}]} W(\ol{k})[\frac{1}{p}].
\]
Thus $\pi_i\TCn(\sT_{\sO_\Cu};\Z_p)_{\free} \otimes_{\Ainf} W(\ol{k})$ is fixed by $G_{K^{ur}}$. The last claim follows from claims (1)-(4) and \cite[Lemma 7.1.2]{Gao}.
\end{proof}

\begin{thm}[Main theorem]\label{mainwithproof}
\begin{itemize}
  \item[(1)] For an $\sO_K$-linear idempotent-complete, small smooth proper stable $\infty$-category $\sT$, there is a natural number $n$ such that for any $i\geq n$, $\pi_i \TCn(\sT/\bS[z];\Z_p)_{\free}$ is a Breuil-Kisin $G_K$-module in the sense of \cite{Gao}. In particular, $\bigl( \widehat{\pi_i\TCn(\sT/\bS[z];\Z_p)_{\free}}\otimes_{\Ainf} W(\Cu) \bigr)^{\widehat{\varphi}\otimes\varphi_{\Cu}=1}$ is a $\Z_p$-lattice of a semi-stable representation, where $\varphi$ is the Frobenius map of $\TCn(\sT/\bS[z];\Z_p)_{\free}$ and $\widehat{\varphi}$ is given in \eqref{power} and $\varphi_{\Cu}$ is the Frobenius of $W(\Cu)$.

  \item[(2)] We assume the conjecture~\ref{kunnethK1} holds. Then for an $\sO_K$-linear idempotent-complete, small smooth proper stable $\infty$-category $\sT$ an integer $i$, there is a $G_K$-equivariant isomorphism
\[
\pi_i \LK K(\sT_{\Cu})_{\free}\otimes_{\Z_p}\Q_p \simeq \bigl( \widehat{\pi_i\TCn(\sT/\bS[z];\Z_p)_{\free}}\otimes_{\Ainf} W(\Cu) \bigr)^{\widehat{\varphi}\otimes \varphi_{\Cu}=1}\otimes_{\Z_p}\Q_p.
\]
In particular, $\pi_i \LK K(\sT_{\Cu})\otimes_{\Z_p}\Q_p$ is a semi-stable representation.

\item[(3)] Assume $K$ is a finite extension of $\Q_p$. For a smooth proper variety $X$ over $\sO_K$, $H^*_\et(X_{\Cu},\Q_p)$ is a semi-stable representation. 
\end{itemize}
\end{thm}
\begin{proof}
 Claim (1) follows from Theorem~\ref{BKGKmodulethm} and \cite[Theorem 7.1.7]{Gao}. We write $\mu=[\epsilon]-1$. We note $\Z_p \to \Ainf[\frac{1}{\mu}]^{\wedge}_p$ is flat. As $\xi | \mu$ and $\mu|\varphi(\mu)$ in $\Ainf$, we have a commutative $G_K$-equivariant diagram:\footnotesize
 \[
 \xymatrix{
&\widehat{\pi_i\TCn(\sT/\bS[z];\Z_p)_{\free}}[\frac{1}{\mu}]^{\wedge}_p \ar[d]_{\simeq}\ar[rr]^-{\widehat{\varphi}}& &\widehat{\pi_i\TCn(\sT/\bS[z];\Z_p)_{\free}}[\frac{1}{\varphi_{\Ainf}(\mu)}]^{\wedge}_p\ar[d]_{\simeq}\\
 \pi_i\LK K(\sT_\Cu)_{\free}\ar@{^{(}-_>}[r]&\pi_i\TCn(\sT_{\sO_\Cu};\Z_p)_{\free}[\frac{1}{\mu}]^{\wedge}_p \ar[r]^-{\eqref{freeTCTP}}&\pi_i\TP(\sT_{\sO_\Cu};\Z_p)_{\free}[\frac{1}{\varphi_{\Ainf}(\mu)}]^{\wedge}_p\ar@{=}[r]&\pi_i\TCn(\sT_{\sO_\Cu};\Z_p)_{\free}[\frac{1}{\varphi_{\Ainf}(\mu)}]^{\wedge}_p\\
 \Z_p \ar@{^{(}-_>}[r]_{} \ar[u]&\Ainf[\frac{1}{\mu}]^{\wedge}_p\ar[u]\ar[r]_{\varphi_{\Ainf}} & \Ainf[\frac{1}{\varphi_{\Ainf}(\mu)}]^{\wedge}_p\ar[u]&
 }
 \]\normalsize
 where all squares are Cartesian by Theorem~\ref{KTCncomp} and the assumption of Conjecture~\ref{kunnethK1}. Since $\Z_p$ is stable under the morphism $\varphi_{\Ainf}$, we have a $G_K$-equivariant inclusion 
 \[
 \pi_i\LK K(\sT_\Cu)_{\free} \subset \bigl( \widehat{\pi_i\TCn(\sT/\bS[z];\Z_p)_{\free}}\otimes_{\Ainf} W(\Cu) \bigr)^{\widehat{\varphi}\otimes \varphi_{\Cu}=1}.
 \]
By Theorem~\ref{KTCncomp} and \cite[Proposition~7.1.4]{Gao}, we obtain the claim (2). The claim (3) follows from Theorem~\ref{commkunnethK1}. 
\end{proof}

\bibliography{bib}
\bibliographystyle{alpha}

%

%

\end{document}